\documentclass{amsart}[12pt]
\usepackage{amsfonts,amssymb,amsmath,amsthm}
\usepackage{enumerate}
\usepackage{amssymb,latexsym}
\usepackage{mathrsfs}
\usepackage{fullpage}
\newtheorem{thrm}{Theorem}[section]
\newtheorem{lem}[thrm]{Lemma}
\newtheorem{prop}[thrm]{Proposition}
\newtheorem{cor}[thrm]{Corollary}
\theoremstyle{definition}

\newtheorem{remark}[thrm]{Remark}
\numberwithin{equation}{section}

\author{N. Demni}
\address{Institut de Recherche en Math\'ematiques de Rennes\\ Universit\'e Rennes 1\\
France}
\email{nizar.demni@univ-rennes1.fr}
\title{Free Jacobi process associated with one projection: local inverse of the flow} 
\begin{document}
\maketitle
\begin{abstract}
We pursue the study started in \cite{Dem-Hmi} of the dynamics of the spectral distribution of the free Jacobi process associated with one orthogonal projection. More precisely, we use Lagrange inversion formula in order to compute the Taylor coefficients of the local inverse around $z=0$ of the flow determined in \cite{Dem-Hmi}. When the rank of the projection equals $1/2$, the obtained sequence reduces to the moment sequence of the free unitary Brownian motion. For general ranks in $(0,1)$, we derive a contour integral representation for the first derivative of the Taylor series which is a major step toward the analytic extension of the flow in the open unit disc.
\end{abstract}

\section{Reminder and motivation}
The free Jacobi process $(J_t)_{t \geq 0}$ was introduced in \cite{Demni} as the large-size limit of the Hermitian matrix Jacobi process (\cite{Dou}). It is built as the radial part of the compression of the free unitary Brownian motion $(Y_t)_{t \geq 0}$ (\cite{Biane}) by two orthogonal projections $\{P,Q\}$: 
\begin{equation*}
J_t := PY_tQY_t^{*}P.
\end{equation*}
In this definition, both families of operators $\{P,Q\}$ and $(Y_t)_{t \geq 0}$ are $*$-free (in Voiculescu's sense) in a von Neumann algebra $\mathscr{A}$ endowed with a finite trace $\tau$ and a unit ${\bf 1}$. When $J_t$ is 
considered as a positive operator valued in the compressed algebra $(P\mathscr{A}P, \tau/\tau(P))$, its spectral distribution $\mu_t$\footnote{In this introductory part, we ommit the dependence of our notations on $\{P,Q\}$.} is a probability distribution supported in $[0,1]$ and the positive real number $\tau(P) \mu_t\{1\}$ encodes the general position property for $\{P, Y_tQY_t^{*}\}$ (\cite{Col-Kem}, \cite{Izu-Ued}). On the other hand, the couple of papers \cite{Dem-Ham-Hmi} and \cite{Dem-Hmi} aim to determine the Lebesgue decomposition of $\mu_t$ when both projections coincide $P=Q$. In particular, when $\tau(P) = 1/2$, a complete description was given in \cite{Dem-Ham-Hmi} (see Corollary 3.3 there and \cite{Izu-Ued} for another proof): at any time $t \geq 0$, $\mu_t$ coincides with the spectral distribution of 
\begin{equation*}
\frac{Y_{2t}+Y_{2t}^{*}+2{\bf 1}}{4}  
\end{equation*}
considered as a positive operator in $(\mathscr{A}, \tau)$ (the spectral distribution of $Y_{2t}$, say $\eta_{2t}$, was described in \cite{Biane1}, Proposition 10). For an arbitrary rank $\tau(P) \in (0,1)$, only the discrete part in the Lebesgue decomposition of $\mu_t$ was determined in \cite{Dem-Hmi} (see Theorem 1.1). As to its absolutely-continuous part with respect to Lebesgue measure in $[0,1]$, it was related to that of the spectral distribution, say $\nu_t$, of the unitary operator 
\begin{equation*}
U_t := RY_tRY_t^{*}, \quad R := 2P-{\bf 1}. 
\end{equation*}
Actually, the density of the former distribution is related to the density of the latter through the Caratheodory extension of the Riemann map of the cut plane $\mathbb{C} \setminus [1,\infty[$ (\cite{Dem-Hmi}, Theorem 1.1). 

The key ingredient leading to this partial description is a flow $\psi_{t}$ so far defined and exploited in an interval of the form $(-1,z_t)$ for some $z_t \in (0,1), t > 0$. When $\tau(P) = 1/2$, $\psi_{t}$ is a one-to-one map from a Jordan domain onto the open unit disc $\mathbb{D}$ and its compositional inverse coincides, up to a Cayley transform, with the Herglotz transform of $\eta_{2t}$. For arbitrary ranks, $\psi_t$ is locally invertible and further information on $\nu_t$ (which in turn provide information on $\mu_t$) necessitates the investigation of the analytic extension of of the local inverse of $\psi_t$ in $\mathbb{D}$. Moreover, it was shown in \cite{Dem-Hmi} that $\nu_t$ converges weakly to a probability measure $\nu_{\infty}$ whose support disconnects as soon as $\tau(P) \neq 1/2$ and it would be interesting to know whether this striking disconnectedness happens or not at a finite time. 

In this paper, we use Lagrange inversion formula and derive the Taylor coefficients of the inverse of $\psi_t$, up to the elementary invertible transformations: 
\begin{equation*}
z \in \mathbb{D} \mapsto s = \frac{1+z}{1-z}, \quad s \mapsto \sqrt{\kappa^2 + (1-\kappa^2)s^2},
\end{equation*}
where $\kappa := \tau(R) = 2\tau(P)-1$.  These coefficients are displayed in corollary \ref{Cor32} of proposition \ref{prop31} below and are given by sign-alternating nested (finite) sums involving Laguerre polynomials and others in the variable $\kappa^2$. In particular, we recover when $\kappa = 0$ the moment sequence of $\eta_{2t}$ while more generally, we can single out from the obtained Taylor series a deformation of the Herglotz transform of $\eta_{2t}$ which is still bounded analytic in $\mathbb{D}$ but no longer extends continuously to the unit circle $\mathbb{T}$ unless $\kappa = 0$. As to the analytic extension of the whole Taylor series, it does not seem accessible (at least for the author) directly from the sums alluded to above due to oscillations. For that reason, we derive a contour integral representation over a circle centered at $\kappa$ for the first derivative of the Taylor series, which is so far valid in a neighborhood of $z=0$. To this end, we use the analytic continuation of the generating series for the Jacobi polynomial outside the interval $[-1,1]$ as well as a special generating series for Laguerre polynomials. The latter series generalizes the Herglotz transform of $\eta_{2t}$ and is expressed through it. Compared to the high dissymmetry arising when $\kappa \neq 0$, this integral representation is a major step toward the extension of the flow in the open unit disc which remains a challenging problem. 

The paper is organized as follows. For sake of completeness, we briefly recall in the next section the relation of the spectral dynamics of $(J_t)_{t \geq 0}$ to those of $(U_t)_{t \geq 0}$ and in particular to the flow $(\psi_t)_{t \geq 0}$. In order to make the paper self-contained, the third section includes the various special functions we use in the remainder of the paper, as well as the computations leading to the Taylor coefficients. The fourth section is devoted to the derivation of the aforementioned integral representation and we close the paper by further developments in relation to the extension of the derived integral in $\mathbb{D}$.  

\section{From the free Jacobi process to the flow}
Though the study of the spectral dynamics of 
\begin{equation*}
J_t = PY_tPY_t^{\star}P
\end{equation*}
was direct when $\tau(P)=1/2$, their study for arbitrary ranks $\tau(P) \in (0,1)$ is rather based on the following binomial-type expansion:
\begin{align*}
\tau[(J_t)^n] = \frac{1}{2^{2n+1}}\binom{2n}{n}  + \frac{\kappa}{2} + \frac{1}{2^{2n}}\sum_{k=1}^n \binom{2n}{n-k}\tau((U_t)^k).
\end{align*}
This expansion has the merit to orient our study to the spectral dynamics of $(U_t)$ which turn out to be easier than those of $J_t$ due to the unitarity of $U_t$. In this respect, let
\begin{equation*}
H_{\kappa, t}(z) := \int_{\mathbb{T}} \frac{w+z}{w-z}\nu_{\kappa, t}(dw) = 1+2\sum_{n \geq 1}\tau(U_t^n)z^n
\end{equation*}
be the Herglotz transform of the spectral distribution $\nu_{\kappa,t}$ of $U_t$. Then the key result proved in \cite{Dem-Hmi} states that there exists a flow $(t,z) \mapsto \psi_{\kappa,t}(z)$ defined in an open set of $\mathbb{R_+} \times [-1,1]$ and such that
\begin{align}\label{Dyn}
[H_{\kappa,\infty}(\psi_{\kappa,t}(z))]^2 - [H_{\kappa,\infty}(z)]^2 = [H_{\kappa, t}(\psi_{\kappa, t}(z))]^2 - [H_{\kappa,0}(z)]^2.
\end{align}   
Here
\begin{equation*}
H_{\kappa,0}(z) = H_0(z) = \frac{1+z}{1-z}
\end{equation*}
is the Herglotz transform of $\nu_{\kappa,0} = \delta_1$ and $H_{\kappa,\infty}$ is that of the weak limit $\nu_{\kappa,\infty}$ of $\nu_{\kappa, t}$ (see section 2 in \cite{Dem-Hmi} for more details on $\nu_{\kappa,\infty}$). 
Equation \eqref{Dyn} was so far used to determine the discrete spectrum of $J_t$ and $\psi_{\kappa, t}$ was expressed as follows: define\footnote{The principal branch of the square root is taken.} 
\begin{equation*}
\alpha: z \mapsto \frac{1-\sqrt{1-z}}{1+\sqrt{1-z}} =  \frac{z}{(1+\sqrt{1-z})^2}.
\end{equation*}
This map extends analytically to $z \in \mathbb{C} \setminus [1,\infty[$ and is one-to-one from this cut plane onto $\mathbb{D}$ whose inverse is 
\begin{equation*}
\alpha^{-1}(z) = \frac{4z}{(1+z)^2}.
\end{equation*}
Recall also from \cite{Biane1} (p.266-269) that the map
\begin{equation*}
\xi_{2t}: z \mapsto \frac{z-1}{z+1}e^{tz}
\end{equation*}
is invertible in some Jordan domain onto $\mathbb{D}$ and that its compositional inverse is the Herglotz transform of $\eta_{2t}$:
\begin{equation*}
K_{{2t}}(z) :=  \int_{\mathbb{T}} \frac{w+z}{w-z}\eta_{2t}(dw) = 1+2\sum_{n \geq 1}\tau(Y_t^n)z^n.
\end{equation*}
If 
\begin{equation*}
s := \frac{1+z}{1-z}, \, z \in \mathbb{D}, \quad a(s) := \sqrt{\kappa^2 + (1-\kappa^2)s^2},
\end{equation*}
then (\cite{Dem-Hmi}, p.283)
\begin{equation*}
\psi_{\kappa,t}(z) = \alpha\left(\frac{a^2}{a^2-\kappa^2}\alpha^{-1}[\xi_{2t}(a(y))]\right).
\end{equation*}
This is a locally invertible map near $z=0$ so that \eqref{Dyn} is equivalent to 
\begin{align*}
[H_{\kappa, t}(z)]^2 - [H_{\kappa,\infty}(z)]^2 = [H_{\kappa,\infty}(\psi_{\kappa,t}^{-1}(z))]^2 + [H_{\kappa,0}(\psi_{\kappa,t}^{-1}(z))]^2
\end{align*}   
near $z=0$. Since $z \mapsto s, s \mapsto a(s)$ and $\alpha$ are invertible transformations the inverting $\psi_{\kappa,t}$ around $z=0$ amounts to the inversion of the map
\begin{equation*}
a \mapsto \frac{a^2}{a^2-\kappa^2}\alpha^{-1}[\xi_{2t}(a)]
\end{equation*}
near $a=1$. This is the main task we achieve in the next section.  

\section{Local inverse of the flow: Lagrange inversion formula}
\subsection{Special functions}
As claimed in the introduction, we list below the special functions occurring in our subsequent computations (see \cite{Erd}, \cite{Man-Sri}, \cite{Rainville} for further details). We start with the Gamma function 
\begin{equation*}
\Gamma(x) = \int_0^{\infty} e^{-u}u^{x-1} du, \quad x > 0, 
\end{equation*}
and the Pochhammer symbol 
\begin{equation*}
(a)_k = (a+k-1)\dots(a+1)a, \quad a \in \mathbb{R}, \, k \in \mathbb{N}, 
\end{equation*}
with the convention $(0)_k = \delta_{k0}$. The latter may be expressed as  
\begin{equation*}
(a)_k = \frac{\Gamma(a+k)}{\Gamma(a)}
\end{equation*} 
when $a > 0$, while
\begin{equation}\label{I1}
\frac{(-n)_k}{k!} = (-1)^k \binom{n}{k}
\end{equation}
if $k \leq n$ and vanishes otherwise. 
Next comes the generalized hypergeometric function defined by the series 
\begin{equation*}
{}_pF_q((a_i, 1 \leq i \leq p), (b_j, 1 \leq j \leq q); z) = \sum_{m \geq 0}\frac{\prod_{i=1}^p(a_i)_m}{\prod_{j=1}^q(b_j)_m}\frac{z^m}{m!}
\end{equation*}
where an empty product equals one and the parameters $(a_i, 1 \leq i \leq p)$ are reals while $(b_j, 1 \leq j \leq q) \in \mathbb{R} \setminus -\mathbb{N}$. 
With regard to \eqref{I1}, this series terminates when at least $a_i = -n \in - \mathbb{N}$ for some $1 \leq i \leq p$, therefore reduces in this case to a polynomial of degree $n$. In particular, the Charlier polynomials are defined by 
\begin{equation*}
C_n(x,a) := {}_2F_0\left(-n, -x; -\frac{1}{a}\right), \quad a \in \mathbb{R} \setminus \{0\}, x \in \mathbb{R}.   
\end{equation*}
When $x \in \mathbb{Z}$ is an integer, a generating function of these polynomials is given by
\begin{equation}\label{GFC}
\sum_{n \geq 0}C_n(x,a)\frac{(au)^n}{n!} = e^{au}\left(1-u\right)^x, \quad |u| < 1. 
\end{equation}
Moreover, the $n$-th Laguerre polynomial with index $\alpha \in \mathbb{R}$ defined by 
\begin{equation}\label{DefL}
L_n^{(\alpha)}(z) := \frac{1}{n!}\sum_{j=0}^n\frac{(-n)_j}{j!}(\alpha+j+1)_{n-j}z^j, 
\end{equation}
is related to the $n$-th Charlier polynomial via:
\begin{equation}\label{CL}
\frac{(-a)^n}{n!}C_n(x,a) = L_n^{(x-n)}(a). 
\end{equation}
When $p=2, q=1$, the Jacobi polynomial $P_n^{a,b}$ of parameters $a,b > -1$ is represented as 
\begin{equation}\label{Jacobi}
P_n^{a,b}(x) := \frac{(a+1)_n}{n!}{}_2F_1\left(-n, n+a+b+1, a+1, \frac{1-x}{2}\right).
\end{equation}

\subsection{Inversion}
Let $t > 0, \kappa \in (-1,1)$ be fixed. The aim of this section is to derive the Taylor coefficients of the inverse of the map $\phi_{\kappa,t}$ defined by 
\begin{equation*}
\phi_{\kappa,t}(z) = \frac{z^2}{z^2-\kappa^2} \alpha^{-1}(\xi_{2t}(z))
\end{equation*}
in a neighborhood of $z=1$. Of course, it is readily checked that $\partial_z\phi_{t, \kappa}(1) \neq 0$ so that $\phi_{t,\kappa}$ is locally invertible there. According to Lagrange inversion formula (see \cite{Man-Sri}, p.354), the Taylor coefficients of its inverse are given by 
\begin{equation*}
a_n(\kappa, t) := \frac{1}{n!}\partial_z^{n-1}\left[\frac{z-1}{\phi_{\kappa,t}(z)}\right]^n|_{z=1}, \quad n \geq 1. 
\end{equation*}
The issue of our computations is recorded in the proposition below:
\begin{prop}\label{prop31}
There exists a set of polynomials $(P_n^{(m)})_{n \geq 1}$ depending on an integer parameter $m \geq 0$ such that 
\begin{align*}
a_n(\kappa, t) = \frac{2}{2^{2n}n}\sum_{k=1}^{n} \binom{2n}{n-k}e^{-kt} \left\{\sum_{m= 0}^{k-1}L_{k-m-1}^{(m+1)}(2kt)\,2^{m}P_n^{(m)}(\kappa^2)\right\}.
\end{align*}
\end{prop}
\begin{proof} 
Set $\epsilon:= \kappa^2$, then we need to expand 
\begin{equation}\label{Lagr}
(z-1)^n \left(1-\frac{\epsilon}{z^2}\right)^n \frac{(1+\xi_t(z))^{2n}}{4^n\xi_t^n(z)}
\end{equation}
around $z=1$. To proceed, we start with 
\begin{align}
\left(1-\frac{\epsilon}{z^2}\right)^n &= \sum_{k=0}^n\binom{n}{k}(-\epsilon)^{k}\frac{1}{(1+(z-1))^{2k}} \nonumber
\\& = \sum_{k=0}^n\binom{n}{k}(-\epsilon)^{k}\sum_{m \geq 0}\frac{(2k)_m}{m!}(1-z)^m \nonumber
\\& :=  \sum_{m \geq 0}(z-1)^mP_n^{(m)}(\epsilon) \label{Poly}
\end{align}
where we set 
\begin{equation*}
P_n^{(m)}(\epsilon) := \frac{(-1)^m}{m!}\sum_{k=0}^n\binom{n}{k}(-\epsilon)^k(2k)_m. 
\end{equation*}

Next, we expand
\begin{align*}
(z-1)^n\frac{(1+\xi_t(z))^{2n}}{\xi_t^n(z)} &= (z-1)^n\sum_{k = -n}^n\binom{2n}{n+k}\xi_t^k
\\& = \binom{2n}{n}(z-1)^n + \sum_{k = 1}^n\binom{2n}{n+k}\left\{\frac{(z-1)^{n+k}}{(1+z)^k}e^{ktz} + (z-1)^{n-k}(1+z)^ke^{-ktz}\right\}
\\& = \binom{2n}{n}(z-1)^n + \sum_{k \in [-n..n]\setminus \{0\}}\binom{2n}{n+k}\frac{(z-1)^{n+k}}{(1+z)^k}e^{ktz} 
\end{align*}
and use \eqref{GFC} together with \eqref{CL} to derive
\begin{align*}
\frac{(z-1)^{n+k}}{(1+z)^k}e^{ktz} & = \frac{e^{kt}}{2^k} (z-1)^{n+k}\sum_{m \geq 0}C_m(-k,-2kt)\frac{(kt(z-1))^m}{m!}
\\& = e^{kt} \sum_{m \geq 0}L_{m}^{(-k-m)}(-2kt)\frac{(z-1)^{m+n+k}}{2^{m+k}}.
\end{align*}
As a result
\begin{align*}
(z-1)^n\frac{(1+\xi_t(z))^{2n}}{\xi_t^n(z)} &= \binom{2n}{n}(z-1)^n + \sum_{k \in [-n..n]\setminus \{0\}} \binom{2n}{n+k}e^{kt} \sum_{m \geq 0}L_{m}^{(-k-m)}(-2kt)\frac{(z-1)^{m+n+k}}{2^{m+k}}.
\end{align*}
From \eqref{I1} and \eqref{DefL}, it is clear that $L_m^{(-m)}(0) = 0$ for all $m \geq 1$ while $L_0^{(0)}(z) = 1$. Hence
\begin{align*}
\frac{(z-1)^n}{[\alpha^{-1}(\xi_t(z))]^n} &= (z-1)^n\frac{(1+\xi_t(z))^{2n}}{4^n\xi_t^n(z)} = \sum_{k = -n}^n \binom{2n}{n+k}e^{kt} \sum_{m \geq 0}L_{m}^{(-k-m)}(-2kt)\frac{(z-1)^{m+n+k}}{2^{2n+m+k}}
\\& =  \sum_{k = 0}^{2n} \binom{2n}{k}e^{(k-n)t}\sum_{m \geq 0} L_{m}^{(n-k-m)}(2(n-k)t)\frac{(z-1)^{m+k}}{2^{n+m+k}}
\\& = \sum_{k = 0}^{2n} \binom{2n}{k}e^{(k-n)t}\sum_{m \geq k} L_{m-k}^{(n-m)}(2(n-k)t)\frac{(z-1)^{m}}{2^{n+m}}
\\ & = \sum_{m \geq 0} \left\{\sum_{k = 0}^{m \wedge 2n} \binom{2n}{k}e^{(k-n)t}L_{m-k}^{(n-m)}(2(n-k)t)\right\}\frac{(z-1)^{m}}{2^{n+m}}.
\end{align*}
Keeping in mind \eqref{Lagr} and \eqref{Poly}, we end up with
\begin{align*}
a_n(\kappa, t) & = \frac{1}{n}\sum_{m=0}^{n-1}\frac{1}{2^{n+m}}P_n^{(n-1-m)}(\epsilon)\left\{\sum_{k = 0}^{2n \wedge m} \binom{2n}{k}e^{(k-n)t} L_{m-k}^{(n-m)}(2(n-k)t)\right\} 
\\& = \frac{1}{n}\sum_{m=0}^{n-1}\frac{1}{2^{n+m}}P_n^{(n-1-m)}(\epsilon)\left\{\sum_{k = 0}^{m} \binom{2n}{k}e^{(k-n)t} L_{m-k}^{(n-m)}(2(n-k)t)\right\} 
\\& = \frac{2}{2^{2n}n}\sum_{k=0}^{n-1} \binom{2n}{k}e^{(k-n)t} \left\{\sum_{m= k}^{n-1}L_{m-k}^{(n-m)}(2(n-k)t)\,2^{n-1-m}P_n^{(n-1-m)}(\epsilon)\right\}
\\& = \frac{2}{2^{2n}n}\sum_{k=1}^{n} \binom{2n}{n-k}e^{-kt} \left\{\sum_{m= n-k}^{n-1}L_{m-n+k}^{(n-m)}(2kt)\,2^{n-1-m}P_n^{(n-1-m)}(\epsilon)\right\} 
\\& = \frac{2}{2^{2n}n}\sum_{k=1}^{n} \binom{2n}{n-k}e^{-kt} \left\{\sum_{m= 0}^{k-1}L_{k-m-1}^{(m+1)}(2kt)\,2^{m}P_n^{(m)}(\epsilon)\right\}.
\end{align*}
\end{proof}
 
\begin{cor}\label{Cor32}
Set $\Phi_{\kappa, t} := \alpha \circ \phi_{\kappa,t}$. Then $\Phi_{\kappa, t}$ is invertible in a neighborhood of $z=1$ and its inverse is given near the origin by
\begin{equation*}
\Phi_{\kappa,t}^{-1}(z) = 1+ \sum_{n \geq 1}\frac{(-1)^n}{n}\left\{\sum_{k=1}^n(-1)^{k} \frac{2n}{n+k}\binom{n+k}{n-k} b_k(\kappa,t)\right\} z^n,
\end{equation*}
where $b_k(\kappa,t) := k 2^{2k}a_k(\kappa,t), k \geq 1$. 
\end{cor}
\begin{proof}
Since $\alpha$ is invertible in $\mathbb{C} \setminus [1,\infty[$ with $\alpha(0) = 0$, then 
\begin{equation*}
\Phi_{\kappa,t}^{-1}(z) = 1+\sum_{n \geq 1}b_n(\kappa,t)\frac{[\alpha^{-1}(z)]^n}{n2^{2n}} 
\end{equation*}
near $z=0$. Differentiating term-wise with respect to $z$ and using the identity 
\begin{equation*}
\partial_z\alpha^{-1}(z) = \frac{1-z}{1+z}\frac{\alpha^{-1}(z)}{z},
\end{equation*}
we get
\begin{equation*}
\partial_z\Phi_{\kappa,t}^{-1}(z) = \frac{1-z}{z(1+z)}\sum_{n \geq 1}b_n(\kappa,t)\frac{[\alpha^{-1}(z)]^n}{2^{2n}}.
\end{equation*}
Now recall the following fact (see \cite{Man-Sri}, p.357): if $(c_n)_{n \geq 0}, (b_n)_{n \geq 0}$ are two sequences of real numbers related by 
\begin{equation}\label{Conv}
b_n = \sum_{k=0}^n\binom{2n}{n-k}c_k,
\end{equation}
then 
\begin{equation*}
\sum_{n \geq 0}b_n\frac{z^n}{4^n} = \frac{1+\alpha(z)}{1-\alpha(z)} \sum_{n \geq 0}c_n[\alpha(z)]^n 
\end{equation*}
whenever both series converge. Equivalently, the relation \eqref{Conv} is invertible with inverse given by (see \cite{Rio}, p.68, Table 2.5, (2)), 
\begin{align}\label{InvRel}
c_0 = b_0, \quad c_n &= \sum_{k=0}^n(-1)^{k+n}\left[\binom{n+k}{n-k} + \binom{n+k-1}{n-k-1}\right] b_k = \sum_{k=0}^n(-1)^{k+n} \frac{2n}{n+k}\binom{n+k}{n-k} b_k, \, n \geq 1,
\end{align}
and 
\begin{equation*}
\sum_{n \geq 0}b_n\frac{[\alpha^{-1}(z)]^n}{4^n} =  \frac{1+ z}{1- z} \sum_{n \geq 0}c_n z^n.
\end{equation*}
The corollary then follows from the substitutions 
\begin{equation*}
b_0 = 0, \quad b_n = b_n(\kappa,t), \, n\geq 1. 
\end{equation*} 
\end{proof}

\begin{remark}
If $\kappa = 0$ then $P_n^{(m)}(0) = \delta_{m0}$ so that 
\begin{equation*}
b_n(0, t) = 2\sum_{k=1}^{n} \binom{2n}{n-k}e^{-kt}L_{k-1}^{(1)}(2kt). 
\end{equation*}
In this case, the inverse of $\Phi_{0,t}$ reduces to
\begin{equation*}
\Phi_{0,t}^{-1}(z) = 1+\sum_{n \geq 1}b_n(0,t)\frac{[\alpha^{-1}(z)]^n}{n2^{2n}} = 1+2\sum_{n \geq 1} \frac{1}{n}e^{-nt}L_{n-1}^{(1)}(2nt)z^{n} 
\end{equation*}
which is nothing else but the Herglotz transform $K_{2t}$ of $\eta_{2t}$ (\cite{Biane1}, p.269).
\end{remark}

\begin{remark}
For general $\epsilon \in [0,1), P_n^{(0)} = (1-\epsilon)^n$ so that the term corresponding to $m=0$ in $a_n(\kappa, t)$ is
\begin{equation*}
 2(1-\epsilon)^n \sum_{k=1}^n \binom{2n}{n-k} e^{-kt}L_{k-1}^{(1)}(2kt) = (1-\epsilon)^nb_n(0, t).
\end{equation*}
Multiplying by $[\alpha^{-1}(z)]^n$ and summing over $n \geq 1$, the preceding remark shows that  
\begin{equation*}
1+2\sum_{n \geq 1}\frac{[(1-\epsilon)\alpha^{-1}(z)]^n}{n 2^{2n}}b_n(0,t)= K_{2t}\left(\alpha[(1-\epsilon)\alpha^{-1}(z)]\right).
\end{equation*}
Since $K_{2t}$ is analytic in $\mathbb{D}$ and since $\alpha$ maps $\mathbb{C} \setminus [1,\infty[$ onto $\mathbb{D}$, then the map 
\begin{equation*}
V_{\kappa, 2t}: z \mapsto K_{2t}(\alpha[(1-\epsilon)\alpha^{-1}(z)])
\end{equation*}
extends analytically to $\mathbb{D}$ with values in the right half-plane $\{\Re(z) > 0\}$. It follows that there exists a probability distribution $\eta_{\kappa,2t}$ on $\mathbb{T}$ such that 
\begin{equation*}
V_{\kappa, 2t} = \int_{\mathbb{T}}\frac{w+z}{w-z}\eta_{\kappa,2t}(dw),
\end{equation*}
and $\eta_{0,2t} = \eta_{2t}$. However, unless $\epsilon = 0$, $V_{\kappa, 2t}$ is no longer continuous in the closed unit disc since the map
\begin{equation*}
z \mapsto \alpha[(1-\epsilon)\alpha^{-1}(z)]
\end{equation*}
is not so in $\mathbb{C}$ due to the presence of the square root function in the definition of $\alpha$. Nonetheless, it is still a bounded holomorphic function in $\mathbb{D}$ since $K_{2t}$ is continuous in $\overline{\mathbb{D}}$ (\cite{Biane1}, Lemma 12).
\end{remark}

\section{An Integral representation}
For any $n \geq 1$, set
\begin{equation}\label{FS}
S_n(\kappa,t) := \sum_{k=1}^n(-1)^{k+n} \frac{2n}{n+k}\binom{n+k}{n-k} b_k(\kappa,t),
\end{equation}
where we recall
\begin{equation*}
b_k(\kappa, t) = 2\sum_{j=1}^{k} \binom{2k}{k-j}e^{-jt} \left\{\sum_{m= 1}^{j-1}L_{j-m-1}^{(m+1)}(2jt)\,2^{m}P_k^{(m)}(\epsilon)\right\}, \quad k \geq 1.
\end{equation*}
For small $|z|$, consider the Taylor series 
\begin{equation*}
M_{\kappa,t}(z) := \sum_{n \geq 2} S_n(\kappa,t) z^n = z\partial_z \Phi_{\kappa,t}^{-1}(z).
\end{equation*}
A major part of this section is devoted to the derivation of the following integral representation:
\begin{prop}
There exists a circle $\gamma_{\kappa}$ centered at $w=\kappa$ and a neighborhood of $z=0$ where 
\begin{align*}
M_{\kappa,t}(z) &= (1-z) \frac{\kappa}{2i\pi}\int_{\gamma_{\kappa}} \frac{[(K_{2t}(y))^2-1]}{[t(K_{2t}(y))^2 + (2-t)][wK_{2t}(y) - \kappa]}\frac{dw}{wR(z,w)}.
\end{align*}
Here 
\begin{equation*}
R(z,w) := \sqrt{(1-z)^2 + 4w^2z},
\end{equation*}
and 
\begin{equation*}
y = y(z,w) := \frac{4z(1-w^2)}{(1+z+R(z,w))^2}. 
\end{equation*}
\end{prop}

\begin{proof}
It consists of four steps corresponding to summation over the indices $\{k,n,j,m\}$ respectively. The interchange of the summation orders is readily justified by the (uniform) estimates given below for the generating series occurring in the proof. In the first step, we make use of the following lemma:  
\begin{lem}
For any $m \geq 0, k \geq 1$, 
\begin{equation}\label{ResInt}
(-1)^mP_k^{(m)}(\epsilon) = \frac{\kappa}{2i\pi}\int_{\gamma_{\kappa}} \frac{w^{m-1}(1-w^2)^{k}}{(w-\kappa)^{m+1}}dw
\end{equation}
where $\gamma_{\kappa}$ is a small circle centered at $w = \kappa$. 
\begin{proof}
If $m \geq 1$ then $(0)_m = 0$ whence
\begin{align*}
(-1)^mP_k^{(m)}(\epsilon) &= \frac{1}{m!}\sum_{l=1}^k\binom{k}{l}(-\epsilon)^l(2l)_m
\\& = \frac{1}{m!} \sum_{l=1}^{k}\binom{k}{l}(-1)^l \kappa^{2l}\frac{(2l+m-1)!}{(2l-1)!}
\\& = \frac{\kappa}{m!} \partial_{w}^{m}\left\{\sum_{l=1}^{k}\binom{k}{l}(-1)^l(w)^{2l+m-1}\right\}_{|w=\kappa}
\\& = \frac{\kappa}{m!} \partial_{w}^{m}[w^{m-1}(1-w^2)^{k}]_{|w=\kappa}.
\end{align*}
The lemma follows from the Cauchy integral formula and from $P_k^{(0)}(\epsilon) = (1-\epsilon)^k$. 
\end{proof}
\end{lem}
Now consider the sum over $k$: 
\begin{align*}
S(n,j) := \sum_{k = j}^n  (-1)^{k+n} \frac{2n}{n+k}\binom{n+k}{n-k} \binom{2k}{k-j} (1-w^2)^k, \quad 1 \leq j \leq n.
\end{align*}
Writing 
\begin{equation*}
\frac{2n}{n+k}\binom{n+k}{n-k} = \binom{n+k}{n-k} + \binom{n+k-1}{n-k-1}
\end{equation*}
then $S(n,j) = f(n,j) - f(n-1,j)$ where 
\begin{align*}
f(n,j) := (-1)^n\sum_{k = j}^n \binom{n+k}{n-k} \binom{2k}{k-j} (w^2-1)^k
\end{align*}
with the convention that an empty sum is zero. Since 
\begin{equation*}
\binom{n+k}{n-k} \binom{2k}{k-j} = \frac{(n+k)!}{(n-k)!(k+j)!(k-j)!},
\end{equation*}
then the index change $k \rightarrow n-k$ yields
\begin{align*}
f_1(n,j) &= \frac{(-1)^{n-j}(1-w^2)^j}{(n-j)!}\sum_{k = 0}^{n-j} \frac{(n-j)!}{(n-j-k)!}\frac{(n+k+j)!}{(k+2j)!} \frac{(w^2-1)^{k}}{k!}
\\& = \frac{(-1)^{n-j}(1-w^2)^j(n+j)!}{(n-j)!(2j)!}\sum_{k = 0}^{n-j} (j-n)_k\frac{(n+j+1)_k}{(2j+1)_k} \frac{(1-w^2)^{k}}{k!}
\\& = \frac{(-1)^{n-j}(1-w^2)^j(2j+1)_{n-j}}{(n-j)!}\sum_{k = 0}^{n-j} (j-n)_k\frac{(n+j+1)_k}{(2j+1)_k} \frac{(1-w^2)^{k}}{k!}
\\& = (-1)^{n-j} (1-w^2)^jP_{n-j}^{2j,0}(2w^2-1)
\end{align*}
where the last equality follows from \eqref{Jacobi}. Besides, the symmetry relation $P_n^{a,b}(z) = (-1)^nP_n^{b,a}(-z)$ entails
\begin{align*}
S(n,j) &= (1-w^2)^j\left\{(-1)^{n-j}P_{n-j}^{2j,0}(2w^2-1) - (-1)^{n-j-1}P_{n-1-j}^{2j,0}(2w^2-1)\right\}
\\& = (1-w^2)^j\left\{P_{n-j}^{0,2j}(1-2w^2) - P_{n-1-j}^{0,2j}(1-2w^2)\right\}.
\end{align*}
Combining \eqref{FS} and the previous Lemma, we get the following representation
\begin{align}\label{Repre}
S_n(\kappa,t) &= \frac{\kappa}{i\pi}\int_{\gamma_{\kappa}}\sum_{j=1}^n[(1-w^2)e^{-t}]^j \sum_{m= 0}^{j-1}L_{j-m-1}^{(m+1)}(2jt)
 \left\{P_{n-j}^{0,2j}(1-2w^2) - P_{n-1-j}^{0,2j}(1-2w^2)\right\}\frac{(-2)^mw^{m-1}}{(w-\kappa)^{m+1}}dw.
\end{align}

In the second step, we fix $j \geq m+1 \geq 1$  and consider the series of Jacobi polynomials: 
\begin{equation*}
\sum_{n \geq j}P_{n-j}^{0,2j}(1-2w^2)z^n = z^j \sum_{n \geq 0}P_n^{0,2j}(1-2w^2)z^n, \quad w \in \gamma_{\kappa}.
\end{equation*}
According to \cite{EKS} (see the discussion p.2), for any fixed $z \in \mathbb{D}$, this series converges uniformly in $w$ on closed subsets of the ellipse $E_{|z|}$ of foci $\{\pm1\}$ and semi-axes 
\begin{equation*}
\frac{1}{2}\left(\frac{1}{|z|} \pm |z|\right). 
\end{equation*}
Since both semi-axes stretches as $|z|$ becomes small, then given an ellipse $E_r, 0 < r < 1$,  the series of Jacobi polynomials above converges absolutely in the disc $\{|z| < r\}$ uniformly on the closure of the domain $D_r$ enclosed by $E_r$. Thus, we fix $r$ and choose $\gamma_{\kappa}$ such that its image under the map $w \mapsto 1-2w^2$ lies in $\overline{D_r}$. Doing so proves the analyticity of the series of Jacobi polynomials above in the variable $w \in E_r$ so that the following equality holds by analytic continuation (see \cite{Sze}, p.69 for real $1-2w^2$):
\begin{equation}\label{GenFunJac1}
z^j \sum_{n \geq 0}P_n^{0,2j}(1-2w^2)z^n = \frac{(4z)^j}{R(z,w)(1+z+R(z,w))^{2j}}, \quad w \in \gamma_{\kappa},
\end{equation}
provided 
\begin{equation*}
R(z,w) = \sqrt{(1-z)^2+4w^2z}
\end{equation*}
does not vanish and is analytic in the variable $z$. This last condition holds true at least for small $|z|$ since  
\begin{equation*}
|[(1-z)^2 +4w^2z] - 1| \leq |z|(|z| + \max_{\gamma_{\kappa}} |1-2w^2|).
\end{equation*}
Similarly, 
\begin{equation}\label{GenFunJac2}
\sum_{n \geq j}P_{n-j-1}^{0,2j}(1-2w^2)z^n = z \frac{(4z)^j}{R(z,w)(1+z+R(z,w))^{2j}}, \quad w \in \gamma_{\kappa}.
\end{equation}

Next comes the third step where we fix the curve $\gamma_{\kappa}$ and $m \geq 1$, and work out the series 
\begin{equation*}
\sum_{j \geq m+1}L_{j-m-1}^{(m+1)}(2jt) \frac{[4ze^{-t}(1-w^2)]^j}{(1+z+R(z,w))^{2j}}, \quad w \in \gamma_{\kappa}.
\end{equation*}
More precisely,
\begin{lem}
For any $m \geq 1$ and $|y| < 1$,
\begin{align}\label{Laguerre}
2^{m+1}\sum_{j \geq m+1}L_{j-m-1}^{(m+1)}(2jt)(e^{-t}y)^j = \frac{(K_{2t}(y))^2-1}{t(K_{2t}(y))^2 + (2-t)}\left[K_{2t}(y)-1\right]^{m}.
\end{align}
\end{lem}
\begin{proof}
The series in the LHS of \eqref{Laguerre} is an instance of equation 1.2 from \cite{Coh} where it is claimed without any detail that it converges in $\mathbb{D}$. This claim can be checked as follows: 
the derivation rule (\cite{Sze}) 
\begin{equation*}
L_{j-m-1}^{(m+1)}(2jt) = \frac{1}{(2j)^m}\partial_t^m L_{j-1}^{(1)}(2jt)
\end{equation*}
 together with the following integral representation (\cite{Gro-Mat}, p.561):
\begin{equation}\label{Int1}
L_{j-1}^{(1)}(2jt) = \frac{1}{2i\pi}\int_{C_0} \left(1+\frac{1}{h}\right)^{j}e^{-2jth} dh
\end{equation}
over a small closed curve $C_0$ around the origin lead to 
\begin{equation}\label{Int2}
L_{j-m-1}^{(m+1)}(2jt) = \frac{(-1)^m}{2i\pi}\int_{C_0} h^m \left(1+\frac{1}{h}\right)^{j}e^{-2jth} dh. 
\end{equation}
When $m=0$, the behavior of $L_{j-1}^{(1)}(2jt)$ as $j \rightarrow \infty$ was analyzed \cite{Gro-Mat} using the saddle point method (see p.561-562). For general $m \geq 1$, the integrand in \eqref{Int2} differs from the one in \eqref{Int1} by the factor $h^m$ which is everywhere analytic and is independent of $j$. According to the saddle point method, the behavior of $L_{j-m-1}^{(m+1)}(2jt)$ as $j \rightarrow \infty$ is the same as that of $L_{j-1}^{(1)}(2jt)$ up to multiplication by $m$-powers of the saddle points (one saddle point for $t \geq 2$ and two conjugate saddle points when $t < 2$). As a matter of fact, formulas (2.57) and (2.60) in \cite{Gro-Mat} show that the series displayed in \eqref{Laguerre} converges for $|y| < 1$. 

Coming into the derivation of the RHS of \eqref{Laguerre}, we specialize equation (1.2) in \cite{Coh} to $b=0, v=m+1, x=2(m+1)t, a = 1/(m+1)$ in order to get:
\begin{align*}
\sum_{j \geq 0} L_{j}^{(m+1)}(2jt+2(m+1)t) (e^{-t}y)^{j+m+1} & = \frac{1}{(1-u)^2+2tu}\frac{(ye^{-t})^{m+1}}{(1-u)^{m}}e^{2(m+1)tu/(u-1)}.
\end{align*}
Here, $u = u_t(z,w) \in \mathbb{D}$ is determined by
\begin{equation*}
e^{-t}y= ue^{2tu/(1-u)} \quad \Leftrightarrow \quad y = ue^{t(1+u)/(1-u)}.
\end{equation*}
Equivalently, 
\begin{equation*}
Z := \frac{u+1}{1-u} 
\end{equation*}
satisfies $\xi_{2t}(Z) = y, \Re(Z) \geq 0$. But since $\xi_{2t}$ is a one-to-one map from the Jordan domain 
\begin{equation*}
\Gamma_{2t} := \{\Re(Z) > 0, \xi_{2t}(Z) \in \mathbb{D}\}
\end{equation*}
onto $\mathbb{D}$ whose composition inverse is $K_{2t}$ (\cite{Biane1}, Lemma 12), then $Z = K_{2t}(y)$ and in turn
\begin{equation*}
u = \frac{Z-1}{Z+1}
\end{equation*}
is uniquely determined in the open unit disc. Substituting 
\begin{eqnarray*}
u & = & (ye^{-t})e^{2ut/(u-1)}  \\ 
\frac{1}{1-u} &=& \frac{K_{2t}(y)+1}{2}
\end{eqnarray*}
proves the lemma.
\end{proof}
According to this lemma,  
\begin{equation*}
2^{m+1}\sum_{j \geq m+1}L_{j-m-1}^{(m+1)}(2jt) \frac{[4ze^{-t}(1-w^2)]^j}{(1+z+R(z,w))^{2j}} = \frac{(K_{2t}(y))^2-1}{t(K_{2t}(y))^2 + (2-t)}\left[K_{2t}(y)-1\right]^m
\end{equation*}
provided that 
\begin{equation*}
y = y(z,w) = \frac{4z(1-w^2)}{(1+z+R(z,w))^2}
\end{equation*}
lies in $\mathbb{D}$, which holds true for $|z|$ small enough. 

Finally, $K_{2t}(y) - 1$ becomes small enough when $|z|$ does since $K_{2t}(0) = 1$. As a result
\begin{equation}\label{Geom}
\sum_{m \geq 0} \frac{w^{m-1}}{(w-\kappa)^{m+1}}\left[1-K_{2t}(y)\right]^m = \frac{1}{w(wK_{2t}(y) - \kappa)}
\end{equation}
in some disc centered at $z=0$. Gathering \eqref{Repre}, \eqref{GenFunJac1}, \eqref{GenFunJac2} and \eqref{Geom}, the proposition is proved. 
\end{proof}
\begin{remark}
If $0 < t \leq 2$ then the map 
\begin{equation*}
z \mapsto t(K_{2t}(y))^2 + (2-t)
\end{equation*}
does not vanish since $K_{2t}$, as a map from the closed unit disc into the right half-plane, takes values in $\sqrt{-1}\mathbb{R}$ only on the unit circle. Otherwise, the range $\Gamma_{2t}$ of $K_{2t}$ does not contain the real $\sqrt{(t-2)/t}$ (see \cite{Biane1}, section 4.2.3). Thus, for any $t > 0$ and any $z$ in the disc evoked in the previous proposition, the map
\begin{equation*}
w \mapsto \frac{1}{[t(K_{2t}(y))^2 + (2-t)]}
\end{equation*}
is holomorphic in the interior of $\gamma_{\kappa}$. This observation together with the splitting $\kappa = (\kappa - wK_{2t}(y)) + wK_{2t}(y)$ lead to
\end{remark}

\begin{cor}
With $\gamma_{\kappa}$ and $z$ as in the previous proposition, 
\begin{align*}
M_{\kappa,t}(z) = (1-z) \frac{1}{2i\pi}\int_{\gamma_{\kappa}} \frac{K_{2t}(y)[(K_{2t}(y))^2-1]}{[t(K_{2t}(y))^2 + (2-t)][wK_{2t}(y) - \kappa]}\frac{dw}{R(z,w)}
\end{align*}
\end{cor}
The second integral representation has the merit to get rid of the singularity of the integrand at $w=0$. It also reduces to $z\partial_z K_{2t}(z) = z\partial_z \Phi_{0,t}^{-1}(z)$ when $\kappa = 0$. In fact, 
\begin{align*}
M_{0,t}(z) &= (1-z) \frac{1}{2i\pi}\int_{\gamma_0} \frac{K_{2t}(y))^2-1}{[t(K_{2t}(y))^2 + (2-t)]}\frac{dw}{w R(z,w)}
\\& = \frac{(1-z)(K_{2t}(y(z,0))^2-1)}{[t(K_{2t}(y(z,0)))^2 + (2-t)]R(z,0)}.
\end{align*}
But $R(z,0) = 1-z$ and $y(z,0)= z$ hence
\begin{align*}
M_{0,t}(z) = \frac{(K_{2t}(z))^2-1)}{[t(K_{2t}(z))^2 + (2-t)]}
\end{align*}
which is the special instance $m=0$ of the RHS of \eqref{Laguerre}: 
\begin{equation*}
\sum_{j \geq 1}L_{j-1}^{(1)}(2jt)z^j = z\partial_z K_{2t}(z).
\end{equation*}

\section{further developments}
So far, the integral representation derived for $M_{\kappa,t}$ is valid in a neighborhood of the origin. With an extra effort, we can prove that the only obstruction toward the analytic extension of $M_{\kappa,t}$ in $\mathbb{D}$ comes from the set of zeros  
\begin{equation*}
wK_{2t}(y(z,w)) - \kappa, \quad \kappa \neq 0,
\end{equation*}
which may or not intersect $\gamma_{\kappa}$ as $z$ varies in $\mathbb{D}$. More precisely, $R(\cdot, w)$ and $y(\cdot, w)$ may be extended to $\mathbb{D}$ after suitably deforming the circle $\gamma_{\kappa}$. For the former, note that the polynomial 
\begin{equation*}
z \mapsto (1-z)^2+ 4\kappa^2z
\end{equation*}
vanishes only on the unit circle and never takes a negative value. Therefore, the range of its restriction to any closed disc in $\mathbb{D}$ remains at a distance $\delta > 0$ from the half-line $]-\infty, 0]$. For those $z$, if 
$\gamma_{\kappa}$ is chosen such that $w^2$ lies in the open disc centered at $\kappa^2$ at a distance $\delta_1 < \delta/4$ then the decomposition
\begin{equation*}
(1-z)^2+4w^2z = [(1-z)^2+4\kappa^2z] + 4\delta_1e^{i\theta}z
\end{equation*}
shows that $(1-z)^2+4w^2z$ does not take values in $]-\infty,0]$. As to the latter, we use a similar reasoning. More precisely, we readily see from
\begin{equation*}
y(z,w) = \frac{4z(1-w^2)}{(1+z+R(z,w))^2} = \frac{1+z - R(z,w)}{1+z + R(z,w)},
\end{equation*}
that $|y(z,w)| < 1$ if and only if 
\begin{equation*}
\Re[(1+z) \overline{R(z,w)}] > 0,
\end{equation*}
that is the inner product of the vectors $(1+z)$ and $R(z,w)$ is positive. But 
\begin{equation*}
R(z, w) = \sqrt{(1+z)^2 - 4(1-w^2)z},
\end{equation*}
and $2\arg(1+z) < \arg(z), |z| < 1$ show that if $w = \kappa \in (-1,1)\setminus \{0\}$ then 
\begin{equation*}
\min\{\arg(1-z), \arg(1+z)\} < \arg R(z,\kappa) < \max\{\arg(1-z), \arg(1+z)\}. 
\end{equation*}
Since $\Re[(1+z)(1-\overline{z})] = 1-|z|^2 > 0$ then $\Re[(1+z) \overline{R(z,\kappa)}] > 0$ and still holds on a small curve around $\kappa$. 


\end{document}